\documentclass[11pt,reqno]{amsart}
\usepackage{amsmath}
\usepackage{amsthm}
\usepackage{amsfonts}
\usepackage{amssymb}
\usepackage{tikz}
\usepackage{float}
\usepackage[all]{xy}
\usepackage{hyperref}

\newcommand{\op}{\ensuremath{^{\mathrm{op}}}}

\newcommand{\Z}{\mathbb{Z}}

\newcommand{\cA}{{\mathcal A}}

\newcommand{\cI}{{\mathcal I}}

\newcommand{\cM}{{\mathcal M}}

\newcommand{\cP}{{\mathcal P}}

\newcommand{\md}{\operatorname{mod}}
\newcommand{\Md}{\operatorname{Mod}}

\newcommand{\Ext}{\operatorname{Ext}}

\newcommand{\Coker}{\operatorname{Cok}}

\usetikzlibrary{matrix,arrows}
\title{Projectively generated $d$-abelian categories are $d$-cluster tilting}

\date{\today}


\author{Sondre Kvamme}

\address{Mathematisches Institut, Universit\"at Bonn, Germany}

\email{sondre@math.uni-bonn.de}

\begin{document}

\newtheorem{Thm}[equation]{Theorem}
\newtheorem{Lemma}[equation]{Lemma}
\newtheorem{Cor}[equation]{Corollary}
\newtheorem{Prop}[equation]{Proposition}

\theoremstyle{definition}
\newtheorem{Def}[equation]{Definition}
\newtheorem{Ex}[equation]{Example}
\newtheorem{Rem}[equation]{Remark}
\newtheorem{Setting}[equation]{Setting}

\thanks{The author thanks Jan Schr\"oer and Gustavo Jasso for comments on a previous version of this paper. The work was made possible by funding provided by the \emph{Bonn International Graduate School in Mathematics}.}


\begin{abstract}
Building on work of Jasso, we prove that any projectively generated $d$-abelian category is equivalent to a $d$-cluster tilting subcategory of an abelian category with enough projectives. This supports the claim that $d$-abelian categories are good axiomatizations of $d$-cluster tilting subcategories.
\end{abstract}

\maketitle

\setcounter{tocdepth}{2}
\numberwithin{equation}{section}

\tableofcontents

\section{Introduction}
The concept of $d$-cluster tilting subcategories was introduced by Iyama in \cite{I1}, and further developed in \cite{I2}, \cite{I3}. It is the natural framework for doing higher Auslander-Reiten theory. A $d$-cluster tilting subcategory $\cM$ is a contravariantly finite, covariantly finite, and generating-cogenerating subcategory of an abelian category $\cA$ satisfying 
\begin{align}\label{eq:1}
\cM= & \{X\in \cA \text{ } |\text{ } \forall i\in \{1,2,\cdots,d-1\} \text{ } \Ext^i_{\cA}(X,\cM)=0\} \\
& \{X\in \cA \text{ } |\text{ } \forall i\in \{1,2,\cdots,d-1\} \text{ } \Ext^i_{\cA}(\cM,X)=0\}
\end{align}
Examples of such categories are given in \cite{HI1}, \cite{HI2}, \cite{IO1}. A problem with this definition is that it is not clear which properties of $\cM$ are independent of the embedding into $\cA$. To fix this, Jasso introduced in \cite{J1} the concept of a $d$-abelian category (see Definition \ref{def:3}), which is an axiomatization of $d$-cluster tilting subcategories. He shows that any $d$-cluster tilting subcategory is $d$-abelian. Furthermore, he also shows \cite[{Theorem 3.20}]{J1} that if $\cM$ is a small projectively generated $d$-abelian category with category of projective objects denoted by $\cP$, such that there exists an exact duality $D\colon \md \cP\to \md \cP\op$, then the image of the fully faithful functor
\[
F\colon \cM\to \md \cP \quad \quad F(X)= \cM(-,X)|_{\cP}
\]
is $d$-cluster tilting in $\md \cP$. Here $\md \cP$ is the category of finitely presented contravariant functors from $\cP$ to $\Md \Z$. In this note we show that the second assumption is unnecessary.

\begin{Thm}\label{thm:1}
Let $\cM$ be a small projectively generated $d$-abelian category, let $\cP$ be the set of projective objects of $\cM$, and let $F\colon \cM\to \md \cP$ be the functor defined by $F(X):= \cM(-,X)|_{\cP}$. Then the essential image
\[
F\cM:= \{M\in \md \cP\text{ }|\text{ }\exists X\in \cM \text{ such that } F(X)\cong M\}
\]
is $d$-cluster tilting in $\md \cP$.
\end{Thm}
We emphasize that almost all of the work towards proving this theorem has been done in \cite{J1}. In fact, by Lemma \ref{lemma:1} the only thing which remains is to show that $F\cM$ is cogenerating and contravariantly finite, and the proof of these properties are straightforward.

\section{Preliminaries} 

We recall the definition of $d$-exact sequences and $d$-abelian categories.

\begin{Def}[{\cite[Definition 2.2]{J1}}]\label{def:1}
Let $\cM$ be an additive category and $f^0\colon X^0\to X^1$ a morphism in $\cM$. A $d$-\emph{cokernel} of $f^0$ is a sequence of maps
\[
(f^1,\cdots f^d): X^1\xrightarrow{f^1} X^2\xrightarrow{f^2}\cdots  \xrightarrow{f^{d-1}}X^{d} \xrightarrow{f^{d}}X^{d+1}
\] 
such that the sequence
\begin{multline*}
0\to \cM(X^{d+1},Z)\xrightarrow{-\circ f^{d}} \cM(X^{d},Z)\xrightarrow{-\circ f^{d-1}} \cdots \\
\cdots \xrightarrow{-\circ f^1}\cM(X^1,Z) \xrightarrow{-\circ f^0}\cM(X^0,Z)
\end{multline*}
is exact for all $Z\in \cM$. Dually, a $d$-\emph{kernel} of a morphism $g^d\colon Y^{d}\to Y^{d+1}$ is a sequence of maps
\[
(g^0,\cdots g^{d-1}): Y^0\xrightarrow{g^0} Y^1\xrightarrow{g^1}\cdots  \xrightarrow{g^{d-2}}Y^{d-1} \xrightarrow{g^{d-1}}Y^d
\] 
such that the sequence
\begin{multline*}
0\to \cM(Z,Y^0)\xrightarrow{g^{0}\circ -} \cM(Z,Y^1)\xrightarrow{g^{1}\circ -} \cdots \\
\cdots \xrightarrow{g^{d-1}\circ -}\cM(Z,Y^d) \xrightarrow{g^d\circ -}\cM(Z,Y^{d+1})
\end{multline*}
is exact for all $Z\in \cM$.
\end{Def}

\begin{Def}[{\cite[Definition 2.4]{J1}}]\label{def:2}
Let $\cM$ be an additive category. A $d$-\emph{exact sequence} is a sequence of maps
\[
X^0\xrightarrow{f^0} X^1\xrightarrow{f^1} \cdots \xrightarrow{f^d} X^{d+1}
\]
such that $(f^0,\cdots ,f^{d-1})$ is a $d$-kernel of $f^d$, and $(f^1,\cdots,f^d)$ is a $d$-cokernel of $f^0$.
\end{Def}  

Recall that $\cM$ is \emph{idempotent complete} if for any idempotent $e\colon X\to X$ in $\cM$ there exists morphisms $\pi\colon X\to Y$ and $i\colon Y\to X$ such that $i\circ \pi = e$ and $\pi\circ i = 1_Y$.

\begin{Def}[{\cite[Definition 3.1]{J1}}]\label{def:3}
A $d$-\emph{abelian category} is an additive category $\cM$ satisfying the following axioms:
\begin{itemize}
	\item[(A0)] $\cM$ is idempotent complete.
	\item[(A1)] Every morphism in $\cM$ has a $d$-kernel and a $d$-cokernel
	\item[(A2)] Let $f^0\colon X^0\to X^1$ be a monomorphism and $(f^1,\cdots f^d)$ a $d$-cokernel of $f^0$. Then the sequence 
	\[
	X^0\xrightarrow{f^0} X^1\xrightarrow{f^1} X^2 \xrightarrow{f^2} \cdots \xrightarrow{f^{d-1}} X^{d} \xrightarrow{f^d} X^{d+1}
	\]
	is $d$-exact.
	\item[(A2\op)] Let $g^d\colon Y^d\to Y^{d+1}$ be an epimorphism and $(g^0,\cdots g^{d-1})$ a $d$-kernel of $g^d$. Then the sequence 
	\[
	Y^0\xrightarrow{g^0} Y^1\xrightarrow{g^1}\cdots  \xrightarrow{g^{d-2}}Y^{d-1} \xrightarrow{g^{d-1}}Y^d\xrightarrow{g^d}Y^{d+1}
	\]
	is $d$-exact.
\end{itemize}
\end{Def}

Recall that $P\in \cM$ is projective if for every epimorphism $f\colon X\to Y$ in $\cM$ the sequence $\cM(P,X)\xrightarrow{f\circ -} \cM(P,Y)\to 0$ is exact. The following results holds for projective objects in $d$-abelian categories.

\begin{Thm}[{\cite[Theorem 3.12]{J1}}]\label{thm:2}
Let $\cM$ be a $d$-abelian category, let $P$ be a projective object in $\cM$, let $f^0\colon X^0\to X^1$ be a morphism in $\cM$, and let $(f^1,\cdots f^d)$ be a $d$-cokernel of $f^0$. Then the sequence
\begin{multline*}
\cM(P,X^0)\xrightarrow{f^0\circ -}\cM(P,X^1)\xrightarrow{f^1\circ -}\cM(P,X^2)\xrightarrow{f^2\circ -} \cdots \\
\cdots \xrightarrow{f^d\circ -}\cM(P,X^{d+1})\to 0
\end{multline*}
is exact.
\end{Thm}

\begin{Def}[{\cite[Definition 3.19]{J1}}]\label{def:4} 
Let $\cM$ be a $d$-abelian category. We say that $\cM$ is \emph{projectively generated} if for every objects $X\in \cM$ there exists a projective object $P\in \cM$ and an epimorphism $f\colon P\to X$.
\end{Def}

Let $\cM$ be a projectively generated $d$-abelian category, let $\cP$ be the category of projective objects of $\cM$, and let $F\colon \cM\to \md \cP$ be the functor $F(X)=\cM(-,X)|_{\cP}$. Theorem \ref{thm:2} tells us that if $(f^1,\cdots f^d)$ is a $d$-cokernel of $f^0$, then the sequence
\[
F(X^0)\xrightarrow{F(f^0)}F(X^1)\xrightarrow{F(f^1)}F(X^2)\xrightarrow{F(f^2)}\cdots \xrightarrow{F(f^{d})}F(X^{d+1})\to 0
\]
is exact in $\md \cP$.

 Parts of the proof that a projectively generated $d$-abelian category is $d$-cluster tilting in $\md \cP$ follows from the following lemma. Note that there is a typo in \cite{J1}; in the lemma they write that $F\cM$ is contravariantly finite, but in the proof they show that it is covariantly finite.   

\begin{Lemma}[{\cite[Lemma 3.22]{J1}}]\label{lemma:1}
Let $\cM$ be a small projectively generated $d$-abelian category, let $\cP$ the category of projective objects of $\cM$, and let $F\colon \cM\to \md \cP$ be the functor defined by $F(X)=\cM(-,X)|_{\cP}$. Also, let
\[
F\cM:=\{M\in \md \cP\text{ }|\text{ }\exists X\in \cM \text{ such that } F(X)\cong M\}
\]
be the essential image of $F$. Then the following holds:
\begin{itemize}
	\item[(i)]$\md \cP$ is abelian;
	\item[(ii)]$F$ is fully faithful;
	\item[(iii)]For all $k\in \{1,\cdots ,d-1\}$ we have
	\[
	\Ext^k_{\md \cP}(F\cM,F\cM)=0;
	\]
	\item[(iv)] We have
	\[
	F\cM= \{M\in \md \cP \text{ } |\text{ } \forall i\in \{1,2,\cdots,d-1\} \text{ } \Ext^i_{\md \cP}(M,F\cM)=0\};
	\]
	\item[(v)] We have
	\[
	F\cM= \{M\in \md \cP \text{ } |\text{ } \forall i\in \{1,2,\cdots,d-1\} \text{ } \Ext^i_{\md \cP}(F\cM,M)=0\};
	\]
	\item[(vi)] $F\cM$ is covariantly finite in $\md \cP$.
\end{itemize}
\end{Lemma}
Since $F\cM$ is obviously generating, it only remains to show that $F\cM$ is cogenerating and contravariantly finite.

\section{Proof of Theorem \ref{thm:1}}
Throughout this section we fix an integer $d\geq 2$, a projectively generated $d$-abelian category $\cM$, and we let $\cP$ denote the category of projective objects in $\cM$. 

\begin{Lemma}\label{lemma:3}
$F\cM$ is cogenerating in $\md \cP$.
\end{Lemma}

\begin{proof}
Let $G\in \md \cP$ be arbitrary. Since $G$ is finitely presented, we can find projective objects $P^0,P^1\in \cM$ and a morphism $\phi\colon F(P^0)\to F(P^1)$ such that $\Coker \phi \cong G$. Since $F$ is full, there exists a morphism $f^0\colon P^0\to P^1$  in $\cM$ such that $F(f^0)=\phi$. Let 
\[
(f^1,\cdots f^d): P^1\xrightarrow{f^1} X^2\xrightarrow{f^2}\cdots  \xrightarrow{f^{d-1}}X^{d} \xrightarrow{f^{d}}X^{d+1}
\]
be a $d$-cokernel of $f^0$. By Theorem \ref{thm:2} we know that the sequence
\[
F(P^0)\xrightarrow{F(f^0)}F(P^1)\xrightarrow{F(f^1)}F(X^2)\xrightarrow{F(f^2)}\cdots \xrightarrow{F(f^{d})}F(X^{d+1})\to 0
\]
is exact. In particular, we have a monomorphism
\[
G\cong \Coker (F(f^0)) \to F(X^2)
\]
This shows that $F\cM$ is cogenerating.
\end{proof}

\begin{Lemma}\label{lemma:4}
$F\cM$ is contravariantly finite in $\md \cP$.
\end{Lemma}

\begin{proof}
Let $G\in \md \cP$ be arbitrary. By Lemma \ref{lemma:3} there exist objects $X^d,X^{d+1}\in \cM$ and an exact sequence
\[
0\to G\xrightarrow{i} F(X^d)\xrightarrow{\phi}F(X^{d+1})
\]
where $\phi=F(f^d)$ since $F$ is full. Let
\[
(f^0,\cdots f^{d-1}): X^0\xrightarrow{f^0} X^1\xrightarrow{f^1}\cdots  \xrightarrow{f^{d-2}}X^{d-1} \xrightarrow{f^{d-1}}X^{d}
\]
be a $d$-kernel of $f^d$. Since $F(f^d)\circ F(f^{d-1})=0$, we get an induced morphism $F(X^{d-1})\xrightarrow{p} G$. We claim that $p$ is a right $F\cM$-approximation of $G$. Let $X\in \cM$ and let $F(X)\xrightarrow{\psi} G$ be an arbitrary morphism in $\md \cP$. Since $F$ is full, the composition $F(X)\xrightarrow{\psi} G\xrightarrow{i} F(X^{d})$ is of the form $F(f)$ for some morphism $f\colon X\to X^d$. Since $f^d\circ f$=0 and 
\[
\cM(X,X^{d-1})\xrightarrow{f^{d-1}\circ -}\cM(X,X^d)\xrightarrow{f^d\circ -}\cM(X,X^{d+1})
\]
is exact, it follows that $f=f^{d-1}\circ g$ for some morphism $g\colon X\to X^{d-1}$. Applying $F$ gives 
\[
i\circ p\circ F(g)= F(f^{d-1})\circ F(g) = F(f) = i\circ \psi
\]
and since $i$ is a monomorphism, we get that $p\circ F(g)= \psi$. This shows that $p$ is a right $F\cM$-approximation, and since $G$ was arbitrary it follows that $F\cM$ is contravariantly finite.  
\end{proof}

\begin{Rem}\label{remark:1}
Let $\cM$ be an injectively cogenerated $d$-abelian category, and let $\cI$ be the category of injective objects in $\cM$. Furthermore, let $G\colon \cM\to (\cI \md)\op$ be the functor given by $G(X):=\cM(X,-)|_{\cI}$. Here $\cI \md$ denotes the category of finitely presented covariant functors from $\cI$ to $\md \Z$. The dual of Theorem \ref{thm:1} tells us that $G$ is a fully faithful functor, $(\cI \md)\op$ is an abelian category, and the essential image 
\[
G\cM:= \{M\in (\cI\md)\op \text{ }|\text{ }\exists X\in \cM \text{ such that } G(X)\cong M\}
\]
is $d$-cluster tilting in $(\cI \md)\op$.
\end{Rem}


\begin{thebibliography}{9}
\bibitem[HI1]{HI1}M. Herschend, O.Iyama
\emph{n-representation finite algebras and twisted fractionally Calabi-Yau algebras}, Bull. London Math. Soc. 43 (2011), 449-466.
\bibitem[HI2]{HI2}M. Herschend, O.Iyama
\emph{Selfinjective quivers with potential and 2-representation-finite algebras}, Compositio Math. 147 (2011), 1885-1920.
\bibitem[I1]{I1}O. Iyama
\emph{Higher-dimensional Auslander-Reiten theory on maximal orthogonal subcategories}. Adv. Math., 210(1):22-50, Mar, 2007
\bibitem[I2]{I2}O. Iyama
\emph{Auslander Correspondence}. Adv. Math., 210(1):51-82, Mar. 2007.
\bibitem[I3]{I3}O. Iyama
\emph{Cluster tilting for higher Auslander algebras}, Adv. Math. 226(1):1-61, Jan. 2011.
\bibitem[IJ]{IJ1} O. Iyama, G. Jasso
\emph{Auslander correspondence for dualizing R-varieties}, Preprint(2016), 18pp., arXiv:1602.00127v2
\bibitem[IO]{IO1} O. Iyama, S. Oppermann
\emph{n-representation-finite algebras and n-APR tilting}, Trans. Amer. Math. Soc. 363 (2011), 6575-6614.
\bibitem[J]{J1}G. Jasso
\emph{n-Abelian and n-exact categories}. Math. Z. (2016): 1-57
\bibitem[W]{W1}C. H. Weibel,
\emph{An introduction to homological algebra}, Cambridge Studies in Adv. Math., Vol 38, Cambridge University Press, 1994
\end{thebibliography}
\end{document}